\RequirePackage{fix-cm}
\documentclass[smallcondensed]{svjour3}     
\smartqed  
\usepackage{siunitx }
\usepackage{subfigure}
\usepackage[utf8]{inputenc}
\usepackage[colorinlistoftodos,prependcaption,textsize=tiny]{todonotes}
\usepackage{algpseudocode}
\usepackage{algorithm}
\usepackage{rotating}
\usepackage{adjustbox}
\usepackage{graphicx}
\graphicspath{ {./images/} }
\usepackage{bigstrut}
\usepackage{multicol}
\usepackage{multirow}
\usepackage{amsmath}
\usepackage[english]{babel}
\usepackage[T1]{fontenc}
\usepackage{hyperref}

\newtheorem{assumption}{Assumption}

\newcommand{\R}{{\mathbf R}}

\DeclareMathOperator*{\argmin}{arg\,min}

\begin{document}

\title{Block Layer Decomposition schemes\\ for training Deep Neural Networks\footnote{This is a preprint of an article published on Journal of Global Optimization. The final authenticated  version is available online  at \url{https://doi.org/10.1007/s10898-019-00856-0}.}}


\titlerunning{Online Block Layer for DNN}        

\author{Laura Palagi\thanks{The author was partially supported by the project
\emph{Distributed optimization algorithms for Big Data} of Sapienza n. RM11715C7E49E89C}, Ruggiero Seccia}

\authorrunning{L. Palagi, R. Seccia} 

\institute{Dip. di Ingegneria informatica automatica e gestionale A. Ruberti
Sapienza - University of Rome
\at
             Via Ariosto 25 - 00185 Roma\\
              Tel.: +39-06-77274081\\
                \email{laura.palagi@uniroma1.it, ruggiero.seccia@uniroma1.it}           
}

\date{Received: date / Accepted: date}

\maketitle

\begin{abstract}
Deep Feedforward Neural Networks' (DFNNs) weights estimation relies on the solution of a very large nonconvex optimization problem that may have many local (no global) minimizers, saddle points and large plateaus. As a consequence, optimization algorithms can be attracted toward local minimizers which can lead to bad solutions or can slow down the optimization process. Furthermore, the time needed to find good solutions to the training problem depends on both the number of samples and the number of variables.
In this work, we show how Block Coordinate Descent (BCD) methods can be applied to improve performance of state-of-the-art algorithms by avoiding bad stationary points and flat regions. We first describe a batch BCD method ables to effectively tackle the  network's depth and then we further extend the algorithm proposing a \textit{minibatch} BCD framework able to scale with respect to both the number of variables and the number of samples by embedding a BCD approach into a minibatch framework. By extensive numerical results on standard datasets for several architecture networks, we show how the application of BCD methods to the training phase of DFNNs permits to outperform standard batch and minibatch algorithms leading to an improvement on both the training phase and the generalization performance of the networks.

\end{abstract}
\section{Introduction}
In this paper we consider the problem of training a Deep Feedforward Neural Network (DFNN) being available a set of training data $\lbrace x_p,
y_p\rbrace$ for $p = 1, \ldots, P$. {According to the Empirical Risk Minimization (ERM) principle}, training a DFNN can be formulated as an unconstrained non convex optimization problem
\begin{equation}\label{eq:ERM}
\min_{w\in \rm I\!R^n}\ f(w)=\frac{1}{P}\sum_{p=1}^PE_p(w)+g(w)
\end{equation}
where the first term represents the average loss, being $E_p$ a measure of the loss on the single sample $p$, and the second term is a regularization term added to improve generalization properties which can also help from an optimization viewpoint by inducing a convexification of the objective function.

This problem is known to be a very hard optimization problem both because is highly \textit{non-convex}, which implies the presence of stationary points that are no global minimizers, and because it is characterized by the presence of plateaus and cliffs which can extremely slow down the speed of convergence of gradient-based optimization algorithms \cite{saddle,palagi}. This peculiar structure together with the large dimension $n$ when considering wide and deep networks and/or the large number of samples $P$ when considering large training set leads to a high computational effort for finding a solution of the optimization problem, being objective function and gradient evaluations the heaviest tasks in the optimization process 

The optimization problem behind the training phase of DFNNs has been principally tackled with two different approaches: \textit{batch} algorithms (a.k.a. \textit{offline} methods), which at each iteration consider the whole  dataset to update the model's weights; \textit{minibatch} algorithms (a.k.a. \textit{online} methods), which at each iteration consider only part of the samples in the training set, a minibatch, to update the weights of the model. The former approach can be effectively used when the number of variables $n$ and the number of samples $P$ in the dataset are not too large. The latter approach, by considering at each iteration only a small subset of all the available data, is more efficient when the dataset is composed by a large number of samples $P$. Although \textit{minibatch} algorithms are less prone to issues when dealing with Big data, they are still affected by the dimension $n$ of the problem, which in DFNN can grow fast. 

An effective solution to solve optimization problems where the number of variables $n$ is very large is represented by \textit{Block Coordinate Descent} (BCD) methods {
\cite{Wright2015,Nesterov2017,Beck2013,Grippof1999}}, which at each iteration update only a subset of the whole variables. In contrast with \textit{minibatch} methods, BCD methods present the opposite behaviour: they are not heavily affected by the number of variables but they still suffer when the number of samples $P$ becomes too large.

{
In this paper, we study how the layered structure of DFNNs can be leveraged to define efficient BCD methods for speeding up the optimization process of these models. In particular, we introduce two different optimization frameworks: 

\begin{itemize}
\item First, following the work done in \cite{Grippo2016} for shallow networks,  we define a general \textit{batch} BCD method for deep networks. We analyze the impact of \textit{backpropagation} procedure on the choice of the blocks of variables and we show how a decomposition approach can help to escape from "bad" attraction regions.
\item Then, exploiting both the variable decomposition, typical of BCD methods, and the sample-wise decomposition, typical of \textit{minibatch methods} we embed the BCD scheme of above into a \textit{minibatch} framework, defining a general \textit{minibatch} BCD framework. Numerical results suggest how this double decomposition leads to a speed up in the solution of the non-convex optimization problem (\ref{eq:ERM}).
\end{itemize}
}

The paper is organized as follows. In Section \ref{Related Works} an overall literature review on both BCD algorithms, \textit{minibatch} algorithms and how these two methods have already been combined is provided. In Section \ref{Preliminaries}  some useful notation is provided and the optimization problem is formulated. In Section \ref{Batch BCD} and \ref{Minibatch BCD}, respectively the \textit{batch} BCD framework and the \textit{minibatch} BCD framework are defined. In these sections, each algorithm is followed by some considerations concerning the optimization strategy implemented and numerical results on some test sets. Finally, in Section \ref{Conclusions}, conclusions and suggestions on further direction of investigations will be discussed.

\section{Related works}\label{Related Works}
\textit{Batch} BCD algorithms are effective methods for large scale optimization problems. By updating a subset of the variables at each iteration, the application of such methods for large scale problems has already been proved to be successful in many applications \cite{BCD_for_LASSO,Nesterov2017,Wright2015,Grippof1999}.

Concerning the application of BCD methods for training NNs, some approaches have already been carried out. 
Extreme Learning Machines (ELMs) \cite{Huang2006,Huang2011} are 
algorithms where the different role played by the output weights, namely weights to the last layer, and hidden weights are considered. Hidden weights are randomly fixed,
while output weights are the tuned by an optimization procedure. 
The great advantage of ELMs is that when a linear unit is used in the output layer, the optimization of the output weights is a Linear Least Square problem. 
From the optimization point of view ELMs so not possess convergence properties to a stationary point.
More sophisticated BCD methods for training Neural Networks have already been proposed. In \cite{Buzzi2001} and in \cite{Grippo2016} globally convergent schemes for training respectively Radial Basis Function (RBF) NNs and MultyLayer Perceptron (MLP) were introduced where the convexity of the objective function when optimizing wrt the output weights was leveraged to improve the optimization process.  These proposed frameworks have been focused on Shallow Neural Networks (SNNs), namely networks with only one hidden layer. At the best of authors' knowledge, this is the first attempt to study how a BCD  scheme could be applied to deep networks which are known to be characterized by harder optimization issues.

Concerning \textit{minibatch} algorithms, the first proposed methods  for training deep networks were Incremental Gradient (IG) \cite{DBLP:journals/corr/Bertsekas15c,Bertsekas:1996:ILS:588880.588920} and Stochastic Gradient  (SG) \cite{Robbins&Monro:1951,Bertsekas2000,Bottou10large-scalemachine}, where the main difference is on how minibatches are picked at each iteration, in an incremental deterministic order in IG or randomly in SG. 
The rate of convergence of these methods is usually slower than standard \textit{batch} methods and depends on the variance in the gradient estimations \cite{Nocedal}. To reduce the variance in the gradient estimation and speed up the optimization process, different approaches were implemented such as \textit{Gradient Aggregation}, where the estimation of the gradient is improved by considering the estimated gradients in the previous iterations. Methods like these are: SVRG \cite{NIPS2013_4937}, SAGA \cite{NIPS2014_5258}, ADAgrad \cite{Duchi:2011:ASM:1953048.2021068} and Adam  \cite{Konur2013}. As already pointed out, these methods are really efficient when the number of data is huge but are not able to scale with respect to the number of variables that for Deep Learning networks can blow up.

The behaviour of \textit{minibatch} BCD methods has already been studied in the strongly convex case where a geometric rate of convergence in expectation has been established \cite{WangB14a} and its effectiveness has been tested in strongly convex sparse problems such as LASSO \cite{zhao} or Sparse Logistic Regression \cite{chauhan17a}. Concerning the application of \textit{minibatch} BCD methods in training Neural Networks, in \cite{bravi2014incremental} a two block decomposition scheme is presented where at each iteration the output weights are exactly optimized using the full batch  while the hidden weights are updated using a minibatch strategy with a step of IG.

This paper provides a general framework for the implementation of batch and online BCD methods in training DFNNs. First, filling the gap left out in \cite{Grippo2016}, we study the effects of variable decomposition from both a computational point of view and an algorithmic performance point of view in \textit{batch } algorithms for deep networks. Afterwards, we introduce a general \textit{minibatch} BCD procedure able to speed up the training process in DFNNs and improve model's generalization properties by leveraging both the advantages of BCD and \textit{minibatch} methods.

\section{Problem Definition} \label{Preliminaries}
Training a DFNN fits in the class of supervised learning, where a 
 a set of data $\lbrace x_p,y_p\rbrace_{p=1}^P$, with $x_p \in {\rm I\!R}^d$ representing the input features and $y_p \in {\rm I\!R}^m$ the corresponding label are given.
A DFNN with $d$ inputs and $m$ outputs is represented by an acyclic oriented network consisting of neurons arranged in $L$ layers connected in a feed-forward way. Each neuron $j$ in a layer $\ell=1,\dots, L$ is characterized by an activation function $g(\cdot)$, that for the sake of simplicity we assume the same for all the neurons, with the only exception of the output layer which has a linear activation function. We introduce the following notation that will be useful in the following.
 
 \begin{itemize}
 \setlength\itemsep{0.33em}
 \item each layer $\ell=1,\dots L$ consists of $N_{\ell}$ neurons, being $N_{0}=d$  and $N_{L}=m$ respectively the input and the output layers;

\item the index set ${\cal L}=\lbrace 1,...,L\rbrace$, where the general index $\ell$ is used to refer to the block of variables $w_\ell$;

\item the matrix  $w_\ell\in {\rm I\!R}^{N_{\ell-1}\times N_\ell}$ representing the weights matrix associated to the  layer $\ell\in \cal L$;

\item  we may refer to the variable $w$ as composed by $L$ blocks of variables, $w=(w_1,\dots,w_L)$;

\item $w_\ell^{ji}$ is the weight  from neuron $i$ in layer $\ell-1$ to neuron $j$ in layer $\ell$;

\item $z_{\ell}^i$ is the output of neuron $i$ in layer $\ell$, $z_{\ell}^i=g(a_{\ell}^i)$, where 
$\displaystyle a_{\ell}^i=\sum_{k=1}^{N_{\ell-1}} w_{\ell-1}^{ik} z_{\ell-1}^k$
\end{itemize}
 
Using this notation, given the sample $x_p$, we have that the output $\tilde{y}_p$  of the DFNN can be written as: 
\begin{equation}\label{eq:output}
         \tilde{y}_p(w)=\tilde{y}(w;x_p)=w_Lg(w_{L-1}g(w_{L-2}\dots g(w_1x_p))).\end{equation}
We considered as error function the convex loss function MSE,  and  we applied a \textit{l-2} norm regularization so that the overall unconstrained problem is continuously differentiable.  With these choices, the optimization problem  \eqref{eq:ERM} can be written as follows:
\begin{equation}\label{optimization_problem}
\underset{w}{\text{min}}f(w)=\frac{1}{P}\sum_{p=1}^P E_p+\rho\Vert w\Vert ^2=\frac{1}{P}\sum_{p=1}^P\|\tilde{y}_p(w)-y_p\|^2 +\rho\Vert w\Vert ^2
\end{equation}
As already said, this problem is highly non-convex with many local minima, saddle points and cliffs which yield to a very hard optimization problem (see \cite{palagi} for a review of recent results on local- global minima issue in DNNs). Furthermore, the dependency on the number of samples and the dimension of the network make the problem even harder slowing down objective function and gradient evaluations. In this setting it is crucial to define optimization frameworks which allows to escape from bad regions and which can scale with respect to the number of samples $P$ and the number of variables $n$.\

\section{Batch  Block Coordinate Decomposition algorithm}\label{Batch BCD}
BCD methods are iterative algorithms where, given the set of all the variables ${\cal W}$, at each iteration $k$ only a suitable subset ${\cal J}^k$, the \textit{ working set}, is selected and only variables belonging to this set are updated by finding a new point $\tilde{w}_i$.
$$w_i^{k+1} = 
\begin{cases}
w_i^k &\qquad \text{if $i\not\in{\cal J}^k$ }       \\
\tilde{w_i} & \qquad \text{ if  $i\in {\cal J}^k$}
\end{cases}$$
In DFNNs, at each iteration, the layered structure of the network can be exploited to select the working set. Indeed, the weights of each layer $w_\ell$ enters in the objective function in a nested way and a "natural" decomposition of $w$ appears that can be fruitful used. 

In this Section, taking steps form the algorithmic scheme proposed in \cite{Grippo2016} for shallow networks, we define an efficient \textit{Batch} (BCD) scheme for the training problem of deep networks and analyze the most critical issues in its characterization. In order to show the effectiveness of the proposed choices we provide extensive numerical results on a set of benchmark datasets to point out the role of variable decomposition when dealing with deep and wide networks.

\subsection{Block Layer Decomposition (BLD) scheme}
We present a BCD scheme, called Block Layer Decomposition (BLD), where blocks of variables are directly defined by the layers of the network $w_\ell$, $\ell=1,\dots,L$. At the basis of this choice is the fact that fixing the  weights of some layer may highlight nice structures of the optimization problem with respect to the other variables, the working set, which makes it easier to solve (cfr \cite{Huang2006,Grippo2016}). 

In order to define the  BLD algorithm, we first report  in Algorithm \ref{alg:Armijo} the Armijo Linesearch procedure that is used in the definition of the BLD algorithm.

  \begin{center}
  \begin{minipage}{.75\linewidth}
    \begin{algorithm}[H]
\caption{Armijo Linesearch}
    \label{alg:Armijo}
    \begin{algorithmic}[1]
\State Given $a>0,\gamma\in (0,1), \delta\in (0,1)$
\State Fix $\alpha=a$
\While {$f(w^k+\alpha d^k)>f(w^k)+\gamma\alpha\nabla f(w^k)^Td^k$}
\State $\alpha=\delta\alpha$
\EndWhile
\State Return $\alpha^k=\alpha$
\end{algorithmic}
    \end{algorithm}
  \end{minipage}
\end{center}

\medskip

 At each iteration $k$ the proposed \textit{batch} BLD method  selects as working set an index $\ell^k\in {\cal L}$ which represents the block of variables belonging to layer $\ell$ and updates only the block $w_\ell^{k}$ by finding a point $\tilde{w}_\ell$ which satisfies some conditions. Then a new index $\ell_{k+1}$ is extracted according to some rule and the same steps are repeated until a stopping criterion, such as a maximum number of iterations or the norm of the gradient below a certain threshold, is met.

In particular, the condition which must be satisfied by the  new point block $\tilde{w}_\ell$ are:
\begin{enumerate}
\item The objective function evaluated at the new point $w^{k+1}=(w_1,...,\tilde{w}_\ell,...,w_L)$  is not worse than the value obtained along the steepest descent direction $d_\ell^k=-\nabla_{w_\ell} f(w^k)$ with the stepsize chosen by an Armijo Linesearch
\begin{equation}\label{cond1}
f(w_1^k,\dots,\tilde{w}_\ell,\dots,w_\ell^k)\leq f(w_1^k,\dots,w_\ell^k+\alpha^kd_\ell^k,\dots,w_\ell^k)
\end{equation}
\item The difference between the objective function evaluated in the  point $w^k$ and in the new one satisfies 
\begin{equation}\label{cond2}
f(w_1^k,\dots,\tilde{w}_\ell,\dots,w_\ell^k)- f(w_1^k,\dots,w_\ell^k,\dots,w_\ell^k)\leq-\sigma(\Vert\tilde{w}_\ell-w_\ell^k\Vert)
\end{equation}
where
$\sigma $  is a forcing function  satisfying
 $$ \lim_{k \to \infty}\sigma(t_k)=0 \implies \lim_{k \to \infty}t_k=0
  $$
\end{enumerate}

The Batch BLD scheme is reported in Algorithm \ref{alg:BCD}.


\begin{algorithm}
\caption{\textit{Batch} Block Layer Decomposition scheme}
\label{alg:BCD}
\begin{algorithmic}[1]
\State Given $\lbrace x_p,y_p\rbrace_{p=1}^P ,{\cal L}=\lbrace 1,...,L\rbrace $
\State Choose $w^0\in\rm I\!R^n$ and set  $k=0$; 
\While {(stopping criterion not met)}
\State Select an index $\ell^k\subseteq {\cal L}$
\State Set $d_{\ell^k}=-\nabla_{w_{\ell^k}}f(w^k)$
\If {$d_{\ell^k} \neq 0$}
\State Compute $\alpha^k $ along $d_{\ell^k}$ with an Armijo Linesearch
\State Find a point $\tilde{w}_{\ell^k}$ such that (\ref{cond1}) and (\ref{cond2}) are satisfied
\Else 
\State ${\tilde{w}}_{\ell^k}={w}^k_{\ell^k} $ 
 
\EndIf
\State Update $w^{k+1}=(w_1^k,\dots,\tilde{w}_{\ell^k},\dots,w_L^k)$
\State Set $k=k+1$
\EndWhile
\end{algorithmic}
\end{algorithm}

Concerning convergence to stationary non-maxima points of the {Batch} BLD framework, we have that it fits within the general decomposition scheme  proposed in Section 7 of \cite{Grippof1999}. To prove convergence we need to introduce the following assumption on the selection of the index set $\ell^k$:

\begin{assumption}[Cyclic updating rule] \label{cyclic}
Blocks of variables $w_\ell$  must be updated in a cyclical order
\end{assumption}

Under this assumption, we can state the following convergence result whose proof, for the sake of completeness, is reported in the Appendix.

\begin{theorem}\label{th:convBLD}
Algorithm BLD with a cyclical selection of the blocks generates a sequence of points $\lbrace w^k\rbrace$ such that
\begin{equation}\label{convergo}
\lim_{k\rightarrow \infty}\nabla f(w^k)=0
\end{equation}
\end{theorem}

Algorithm \ref{alg:BCD} must be further characterized depending on how the index $\ell ^k$ is chosen and how  the new point $\tilde{w}_\ell$ is computed. These choices affect both the convergence properties of the algorithm and its computational efficiency. In the following, we discuss  the possible choices that satisfy assumptions to guarantee convergence of BLD and allow to save computations, mainly due to the reduction of the needed time for objective function and gradient evaluations.

\subsection*{Selection of the \textit{working set} $\ell^k$} \label{sec: J}
Before providing general rules on how to choose the index $\ell^k$, we need to recall how the \textit{backpropagation} procedure works in order to comprehend the crucial role played by the working set definition in the BLD scheme.

In DFNNs, the main effort in computing $\nabla f(w)$  stays in evaluating $\sum_{p=1}^P \nabla_w E_p$.  Thanks to the chain rule, each element $\nabla_w E_p$ is computed  as
\begin{equation} \label{eq: backprop1}
\frac{\partial E_p(w)}{\partial w_\ell^{ji}}=\frac{2}{P}z_{\ell-1}^i\delta_\ell^j 
\end{equation}
where the $z_{\ell}$  are obtained by forward propagation of the input and the $\delta_\ell$ 
by backward propagation throughout the layers of the error $e=\tilde{y}_p-y_p$
\begin{equation} \label{eq: backprop2}
\delta_L^j=eg'(a_L)\ \qquad \delta_\ell^j=\sum_{k=1}^{N_{\ell+1}}\delta_{\ell+1}^k w_{\ell+1}^{kj}g'(a_\ell^j) \qquad l\leq L-1
\end{equation}
where  $g'$ is the partial derivative of $g$.
\begin{figure}
\centering
\includegraphics[scale=0.35]{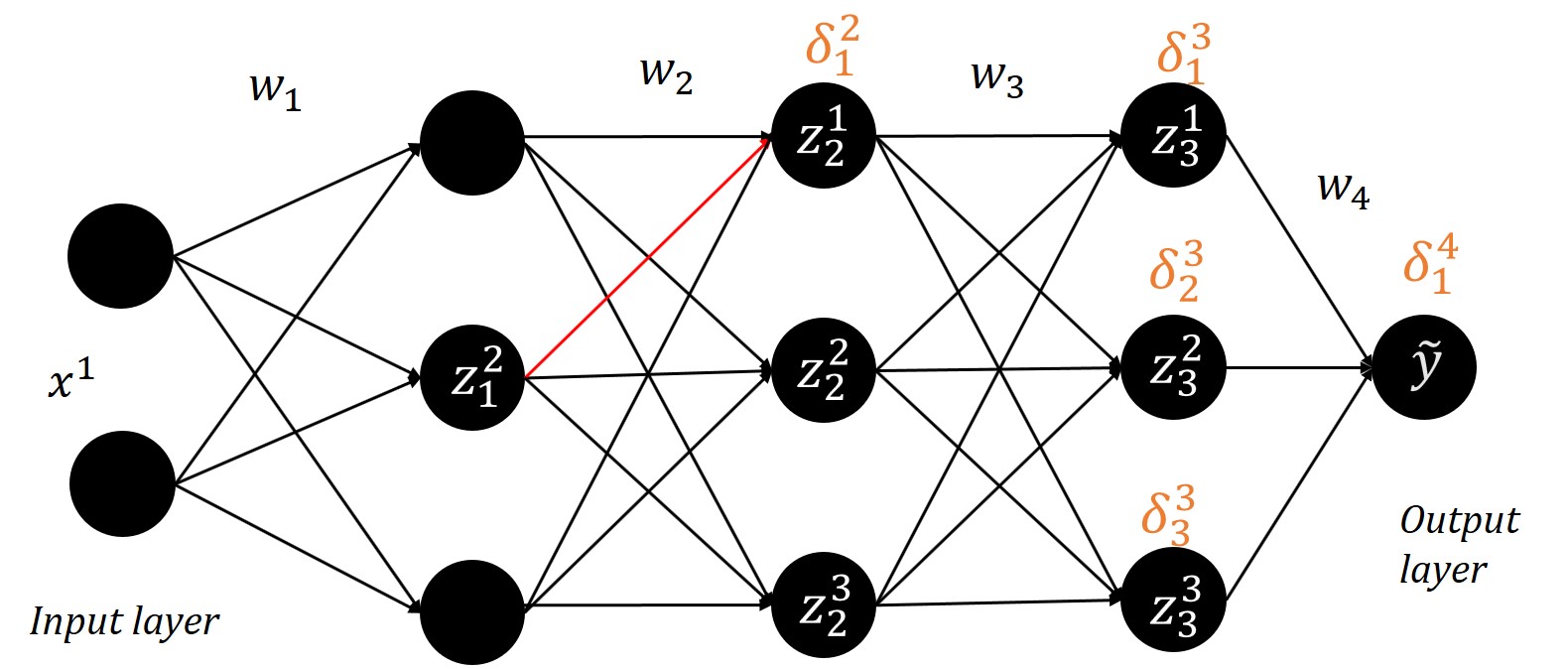}
\caption{Example of how \textit{backpropagation} works to get the derivative wrt a hidden weight}
\end{figure}
 
This procedure is called \textit{backpropagation} because, 
to get derivatives of  variables in layer $\ell$, we need to evaluate the $\delta$s of all the following layers, starting from the last one $\delta_L$ and backpropagating the $\delta$s until $\delta_{\ell}$, being the general  $\delta_\ell$ depending on $\delta_{\ell+1}$. Then, if the full gradient $\nabla f(w^k)$ must be evaluated, as in standard gradient based methods, forward and backward propagation involves all the layers of the network. On the other hand, when only the gradient wrt a block $\ell$ is needed, as in BLD scheme, the  propagation steps go only over \textit{part} of the network, reducing the computational effort.

In order to save computations in the evaluation of the gradient, the chosen block layer decomposition turns out to be the more efficient. Indeed, looking at (\ref{eq: backprop1}) and (\ref{eq: backprop2}), it is clear that blocks of variables defined with interlacing weights over the layers would imply the forward and backward propagation along the full network for the computations of $\delta$s. Furthermore, we highlight that also a definition of blocks of variables by neurons, as done in \cite{Grippo2016} for shallow networks, turns out to be inefficient, since to update the general neuron $j$ at layer $\ell$ we need to compute $\delta_\ell$ which depends by all the $\delta$s of the following layers which would be computed but not further used. Another possible choice for exploiting the backpropagation procedure consists in updating at iteration $k$ the weights of layer $\ell^k$  and all the successive layers $\ell^k+1, \dots, L$, being the following $\delta$s already computed to evaluate $\delta_\ell$.  This blocks selection rule would be more efficient since it would imply the utilization of all the $\delta$s computed at each iteration. However, this procedure leads to an unbalanced decomposition update because the last layers are optimized  more often  than the first ones. Extensive numerical results (not reported here for the lack of space) shows a worst performance as final value of the objective function. This behaviour is arguably due to the fact that this  strategy gives more importance to some layers than others making the algorithm converge to very poor regions of the problem and severally harming the overall training process.

As regard the potential strategies for choosing the index $\ell^k$, Assumption (\ref{cyclic})
 can be ensured by using several block selection rules such as:
\renewcommand\labelitemi{-}
\begin{itemize}
\item \textbf{\textit{Forward} selection}. Layers are selected according to the forward propagation $\lbrace 1,2,...,L\rbrace$.
\item \textbf{\textit{Backward} selection}. layers are selected according to the backward propagation $\lbrace L,L-1,...,1\rbrace$.

\item \textbf{Random without replacement.} The blocks are selected with a cyclic order that changes after all the blocks have been updated. This corresponds to a cyclic rule with random reshuffling. In this case  
\end{itemize}

\subsection*{Block $\tilde{w}_\ell$ update }
{Once  block $\ell^k$ is selected, a point $\tilde w_{\ell^k}$ satisfying conditions (\ref{cond1}) (\ref{cond2}) must be determined.
First we note that, similarly to Extreme Learning Machine \cite{Huang2006,Huang2011}, when the last block $\ell=L$ is selected, the optimization problem becomes a linear least square problem (LLSQ) with an additional \textit{l-2} regularization term which makes the subproblem in the variables $w_L$ strictly convex. As a consequence, an update $\tilde w_L$  satisfying conditions
\ref{cond1}, \ref{cond2} can be easily found either by solving the subproblem in closed form or by the application of an iterative method such as the Conjugate Gradient algorithm \cite{NoceWrig06}. Note that since the optimization wrt $w_L$ is quadratic strictly convex, the update $\tilde{w}_L^{k+1}=\argmin_{w_L}f(w^k)$ is always bounded to guarantee conditions (\ref{cond1})(\ref{cond2}) (see \cite{Grippof1999}).

When a block different from the last one is selected, $\ell\neq L$, the optimization problem becomes non-convex and finding good updates  $\tilde w_\ell$ may becomes harder. We first note that the point returned by the  Armijo Linesearch along the steepest descent direction
$$\tilde w_{\ell^k}=w_{\ell^k}^k-\alpha^k\nabla_{w_{\ell^k}} f(w^k)$$
satisfies the two conditions (\ref{cond1}) (\ref{cond2}).
However a single iteration of the gradient method for each $\ell^k$ can slow down the overall procedure. More in general, it is easy to prove that a finite number of steps of a gradient method with Armijo Linesearch guarantees conditions (\ref{cond1}) (\ref{cond2}) (see Appendix for more details), so that $\tilde w_{\ell^k}$ can be obtained by iterating a gradient method with an Armijo Linesearch.  However, as suggested in \cite{Grippo2016}, since the number of variables in NNs may be very high, methods which approximate second-order information and do not suffer too much of the \textit{curse of dimensionality} would perform remarkably better than gradient methods. So, at each iteration an optimization method such as a limited memory Quasi-Newton method \cite{NoceWrig06} like LBFGS can be applied to solve the subproblem in the variables $w_{\ell^k}$ to get a trial point $w^*_{\ell^k}$. If conditions \ref{cond1}, \ref{cond2} are satisfied, $\tilde w_{\ell^k}=w^*_{\ell^k}$, otherwise $\tilde w_{\ell^k}$ is set to the point obtained by an  Armijo Linesearch along the steepest descent direction.


\subsection{Implementation and performance of a BLD method}

{For the experimental results, we implemented a  BLD algorithm with the following settings:
\begin{itemize}
\item a \textit{backward cyclic order} selection rule of the layer;
\item choice of $\tilde w_\ell^k$ by the application of a LBFGS method with increasing accuracy $\varepsilon$ and limited number of iterations with a check on conditions (\ref{cond1}) (\ref{cond2}) as described above.
\end{itemize}
}

\begin{algorithm}
    \caption{Backward Block Layer Decomposition (B$^2$LD)}
    \label{alg:BLD}
    \begin{algorithmic}[1]
\State Given $\lbrace x_p,y_p\rbrace_{p=1}^P $
\State Choose $w^0\in\rm I\!R^n$ and set  $k=0$; set $\epsilon>0,\delta\in(0,1)$.
\While {stopping criterion not met}
\For  {$l=L,...,1$}
\State Find $\tilde{w}_\ell$ s.t. $\|\nabla_{w_\ell} f(\tilde{w}_\ell)\|\le \epsilon$
\If {$\tilde{w}_\ell$ satisfies (\ref{cond1}) (\ref{cond2})}
\State $w^{k+1}=(w_1^k,\dots,\tilde{w}_\ell,\dots,w_L^k)$
\Else 
\State $w^{k+1}=(w_1^k,\dots,w_\ell^k-\alpha^k \nabla_{w_\ell} f(w^k),\dots,w_L^k)$
\EndIf
\State $k=k+1$
\EndFor
\State $\epsilon=\epsilon\delta$
\EndWhile
\end{algorithmic}
\end{algorithm}

We observe that we applied an LBFGS method also to solve the LLSQ problem in the variables of the last block $w_L$. This choice is due to the observation in the numerical results that forcing a high accuracy in the optimization of the last layer from the early stages seemed to get worse results in terms of quality of the solution (value of the objective function and test error as well).
 
We compared B$^2$LD performance against an LBFGS method applied  to the whole optimization problem \ref{optimization_problem}. The initial point $w^0$ has been set to the same randomly chosen value for both the two methods and performance are compared on the results obtained within 10 runs.  The regularization parameter $\rho$ was set equal to {$\rho=10^{-3}/n$} and the sigmoid function was used as activation function. No hyperparameters tuning was carried since the focus is on studying optimization algorithms performance more than finding the smaller generalization error. 

The following stopping criteria  were set: i) $\|\nabla f(w^k)\|\le  10^{-3}$; ii)  $\displaystyle f_{tol}=\frac{f(w^k)-f(w^{k+1})}{\max\lbrace f(w^k),1 \rbrace}\le 10^{-4}$; iii) computational time exceeds 150 seconds. Note that for B$^2$LD condition ii) must be satisfied with respect to each block in order to stop computations. Moreover, computational times were checked only every 30 iterations of LBFGS method that is why in Table \ref{tab:results_1x50} Cpu Times can be higher than 150 seconds.

Both the two algorithms were implemented in Python version 3.6, using the package \textit{numpy} for numerical calculus and \textit{scipy} for the optimization algorithms using a Intel(R) Core(TM) i7-870 CPU 2.93 GHz.\\

The two algorithms have been compared over five different network architectures in order to analyze how the algorithms perform when dimensions of the problems increase with respect to two different parameters: width and depth of the network. The characteristic of the different architectures are reported in Table \ref{tab:dataset}. For the sake of simplicity in most cases we assume equal numbers of neurons per layer $N_\ell=N$, so that each network is identified by the notation $[L\times N]$, with the exception of the network $[200,50,200]$, where the values represent the number of neurons per the three layers. Comparison were carried over seven  different datasets publicly available at {\tt  https://sci2s.ugr.es/keel/index.php} and {\tt https://archive.ics.uci.edu/ml/index.php }.
The number of samples in the training and test sets and the number of features per dataset are reported in Table \ref{tab:dataset} along with the number of variables per each optimization problem.


\begin{table}[htbp]
  \centering
      \caption{Datasets and Networks Description}
  \begin{adjustbox}{max width=\textwidth}
    \begin{tabular}{l|ccc|ccccc}
          &       &       &       & \multicolumn{5}{c}{\# Variables in the network} \\
    Dataset & \multicolumn{1}{l}{\# Train} & \multicolumn{1}{l}{\# Test} & \multicolumn{1}{l|}{\# Features} & \multicolumn{1}{l}{[1 x 50]} & \multicolumn{1}{l}{[3 x 20]} & \multicolumn{1}{l}{[200,50,200]} & \multicolumn{1}{l}{[5 x 50]} & \multicolumn{1}{l}{[10 x 50]} \\
    \hline
    Mv    & 32614 & 8154  & 13    & 700   & 1080  & 22800 & 10700 & 23200 \\
    Bikes Sharing & 13903 & 3476  & 59    & 3000  & 2000  & 32000 & 13000 & 25500 \\
    Bejing PM 2.5 & 33406 & 8351  & 48    & 2450  & 1780  & 29800 & 12450 & 24950 \\
    CCPP  & 7654  & 1913  & 5     & 300   & 920   & 21200 & 10300 & 22800 \\
    Ailerons & 11000 & 2750  & 41    & 2100  & 1640  & 28400 & 12100 & 24600 \\
    California & 16512 & 4128  & 9     & 500   & 1000  & 22000 & 10500 & 23000 \\
    Bank  & 36168 & 9042  & 40    & 2050  & 1620  & 28200 & 12050 & 24550 \bigstrut[b]\\
    \hline
    \end{tabular}%

    \end{adjustbox}
  \label{tab:dataset}%
\end{table}%

These numerical experiments aim to show that using BLD may help in obtaining better results in term of the final objective function (regularized training error) without deteriorating the generalization performance. We analyze both the best performance over the 10 runs and the average behaviour.

In Table \ref{tab:results_1x50} the detailed results corresponding to the best run of each of the two algorithms are reported. For each network architecture and for each problem we report: the best solution found, the corresponding norm of the gradient and the time needed to find it. 

Analyzing the numerical results, we observe that BLD method is able to find better values of the objective function with a smaller value of the norm of the gradient. It seems that BLD does not get stuck in poor solutions as it happens to the standard LBFGS method. 
 This issue is particularly evident when the number of layers increases. Indeed,  the results on the deepest network $[10 \times 50 ]$
 show that LBFGS stops very quickly at a very poor solution with a small norm of the gradient whereas BLD is able to continue optimization.
 This may be due also the \textit{vanishing gradient} effect which may affect a full gradient method.

\begin{table}[htbp]
  \centering
  \caption{Best result returned over 10 runs by the B$^2$LD and LBFGS methods}
  
   \begin{adjustbox}{max width=\textwidth}
    
   \begin{tabular}{l|rr|rr|rr}

    \multicolumn{1}{p{4.215em}|}{Network } & \multicolumn{2}{c|}{Objective Function} & \multicolumn{2}{c|}{Norm of the Gradient} & \multicolumn{2}{c}{Cpu Time} \bigstrut[b]\\
\cline{2-7}     [1 x 50] & \multicolumn{1}{l}{B$^2$LD} & \multicolumn{1}{l|}{LBFGS} & \multicolumn{1}{l}{B$^2$LD} & \multicolumn{1}{l|}{LBFGS} & \multicolumn{1}{l}{B$^2$LD} & \multicolumn{1}{l}{LBFGS} \bigstrut\\
    \hline
    Ailerons & 2.37$\times 10^{-3}$ & \textbf{2.32$\times 10^{-3}$} & 9.98$\times 10^{-4}$ & \textbf{4.03$\times 10^{-4}$} & \textbf{0.18} & 3.04 \bigstrut[t]\\
    Bank Marketing & 7.92$\times 10^{-5}$ & \textbf{3.67$\times 10^{-5}$} & \textbf{4.96$\times 10^{-4}$} & 6.06$\times 10^{-4}$ & \textbf{13.91} & 29.77 \\
    Bejing Pm25 & \textbf{4.79$\times 10^{-3}$} & 5.32$\times 10^{-3}$ & \textbf{8.57$\times 10^{-4}$} & 2.25$\times 10^{-3}$ & 9.16  & \textbf{8.43} \\
    Bikes Sharing & \textbf{2.04$\times 10^{-3}$} & 3.01$\times 10^{-3}$ & \textbf{7.12$\times 10^{-4}$} & 4.10$\times 10^{-3}$ & 16.44 & \textbf{11.58} \\
    California & 2.00$\times 10^{-2}$ & 2.00$\times 10^{-2}$ & \textbf{2.98$\times 10^{-4}$} & 3.34$\times 10^{-3}$ & \textbf{1.50} & 2.99 \\
    CCPP  & \textbf{3.35$\times 10^{-3}$} & 3.53$\times 10^{-3}$ & \textbf{5.88$\times 10^{-4}$} & 6.87$\times 10^{-4}$ & \textbf{0.08} & 0.91 \\
    Mv    & 2.93$\times 10^{-4}$ & \textbf{2.68$\times 10^{-4}$} & 7.18$\times 10^{-4}$ & \textbf{4.90$\times 10^{-4}$} & \textbf{20.90} & 23.78 \bigstrut[b]\\
\hline\hline
     [3 x 20] & \multicolumn{1}{l}{B$^2$LD} & \multicolumn{1}{l|}{LBFGS} & \multicolumn{1}{l}{B$^2$LD} & \multicolumn{1}{l|}{LBFGS} & \multicolumn{1}{l}{B$^2$LD} & \multicolumn{1}{l}{LBFGS} \bigstrut[b]\\
    \hline
    Ailerons & 2.33$\times 10^{-3}$ & \textbf{2.24$\times 10^{-3}$} & 9.29$\times 10^{-4}$ & \textbf{4.88$\times 10^{-4}$} & \textbf{3.50} & 3.62 \bigstrut[t]\\
    Bank Marketing & 1.09$\times 10^{-4}$ & \textbf{3.91$\times 10^{-5}$} & 8.09$\times 10^{-4}$ & \textbf{7.66$\times 10^{-4}$} & 59.19 & \textbf{17.45} \\
    Bejing Pm25 & \textbf{4.45$\times 10^{-3}$} & 7.60$\times 10^{-3}$ & \textbf{9.48$\times 10^{-4}$} & 1.36$\times 10^{-3}$ & 6.24  & \textbf{6.04} \\
    Bikes Sharing & \textbf{1.88$\times 10^{-3}$} & 2.81$\times 10^{-3}$ & \textbf{8.37$\times 10^{-4}$} & 2.36$\times 10^{-3}$ & \textbf{10.76} & 16.54 \\
    California & \textbf{2.02$\times 10^{-2}$} & 2.25$\times 10^{-2}$ & \textbf{4.10$\times 10^{-4}$} & 2.28$\times 10^{-3}$ & 2.44  & \textbf{1.56} \\
    CCPP  & 3.47$\times 10^{-3}$ & \textbf{3.35$\times 10^{-3}$} & 3.89$\times 10^{-4}$ & \textbf{3.38$\times 10^{-4}$} & 1.51  & \textbf{1.17} \\
    Mv    & 2.92$\times 10^{-4}$ & \textbf{2.17$\times 10^{-4}$} & 7.71$\times 10^{-4}$ & \textbf{3.70$\times 10^{-4}$} & \textbf{23.86} & 37.44 \bigstrut[b]\\
\hline\hline
     [200.50.200] & \multicolumn{1}{l}{B$^2$LD} & \multicolumn{1}{l|}{LBFGS} & \multicolumn{1}{l}{B$^2$LD} & \multicolumn{1}{l|}{LBFGS} & \multicolumn{1}{l}{B$^2$LD} & \multicolumn{1}{l}{LBFGS} \bigstrut[b]\\
    \hline
    Ailerons & \textbf{2.33$\times 10^{-3}$} & 2.56$\times 10^{-3}$ & \textbf{9.09$\times 10^{-4}$} & 1.58$\times 10^{-3}$ & \textbf{17.93} & 36.77 \bigstrut[t]\\
    Bank Marketing & \textbf{1.93$\times 10^{-4}$} & 2.31$\times 10^{-4}$ & \textbf{1.61$\times 10^{-3}$} & 4.98$\times 10^{-3}$ & 184.34 & \textbf{151.09} \\
    Bejing Pm25 & \textbf{4.59$\times 10^{-3}$} & 5.76$\times 10^{-3}$ & \textbf{8.44$\times 10^{-4}$} & 2.59$\times 10^{-3}$ & 137.09 & \textbf{113.77} \\
    Bikes Sharing & \textbf{2.01$\times 10^{-3}$} & 4.00$\times 10^{-3}$ & \textbf{5.81$\times 10^{-4}$} & 3.12$\times 10^{-3}$ & 146.89 & \textbf{132.06} \\
    California & \textbf{1.98$\times 10^{-2}$} & 2.31$\times 10^{-2}$ & 1.14$\times 10^{-2}$ & \textbf{5.46$\times 10^{-3}$} & 17.37 & \textbf{15.41} \\
    CCPP  & \textbf{3.50$\times 10^{-3}$} & 3.70$\times 10^{-3}$ & \textbf{7.04$\times 10^{-4}$} & 8.17$\times 10^{-4}$ & \textbf{6.67} & 16.39 \\
    Mv    & \textbf{2.49$\times 10^{-4}$} & 9.73$\times 10^{-4}$ & \textbf{7.84$\times 10^{-4}$} & 5.59$\times 10^{-3}$ & \textbf{145.32} & 185.38 \bigstrut[b]\\
\hline\hline
     [5 x 50] & \multicolumn{1}{l}{B$^2$LD} & \multicolumn{1}{l|}{LBFGS} & \multicolumn{1}{l}{B$^2$LD} & \multicolumn{1}{l|}{LBFGS} & \multicolumn{1}{l}{B$^2$LD} & \multicolumn{1}{l}{LBFGS} \bigstrut[b]\\
    \hline
    Ailerons & \textbf{2.37$\times 10^{-3}$} & 4.89$\times 10^{-3}$ & \textbf{8.15$\times 10^{-4}$} & 5.49$\times 10^{-3}$ & 12.57 & \textbf{10.52} \bigstrut[t]\\
    Bank Marketing & 2.02$\times 10^{-4}$ & \textbf{9.19$\times 10^{-5}$} & \textbf{9.41$\times 10^{-4}$} & 1.43$\times 10^{-3}$ & \textbf{110.51} & 152.88 \\
    Bejing Pm25 & \textbf{5.03$\times 10^{-3}$} & 7.26$\times 10^{-3}$ & \textbf{5.38$\times 10^{-4}$} & 1.34$\times 10^{-3}$ & 42.72 & \textbf{27.26} \\
    Bikes Sharing & \textbf{2.07$\times 10^{-3}$} & 3.46$\times 10^{-2}$ & \textbf{5.70$\times 10^{-4}$} & 5.06$\times 10^{-3}$ & 74.30 & \textbf{1.81} \\
    California & \textbf{2.05$\times 10^{-2}$} & 5.47$\times 10^{-2}$ & 2.92$\times 10^{-3}$ & \textbf{1.57$\times 10^{-3}$} & 12.56 & \textbf{1.46} \\
    CCPP  & \textbf{3.79$\times 10^{-3}$} & 4.45$\times 10^{-3}$ & \textbf{4.54$\times 10^{-4}$} & 1.54$\times 10^{-3}$ & \textbf{7.33} & 8.61 \\
    Mv    & \textbf{3.59$\times 10^{-4}$} & 7.76$\times 10^{-4}$ & \textbf{6.12$\times 10^{-4}$} & 8.38$\times 10^{-3}$ & \textbf{131.44} & 150.55 \bigstrut[b]\\
\hline\hline
     [10 x 50] & \multicolumn{1}{l}{B$^2$LD} & \multicolumn{1}{l|}{LBFGS} & \multicolumn{1}{l}{B$^2$LD} & \multicolumn{1}{l|}{LBFGS} & \multicolumn{1}{l}{B$^2$LD} & \multicolumn{1}{l}{LBFGS} \bigstrut[b]\\
    \hline
    Ailerons & \textbf{2.65$\times 10^{-3}$} & 1.31$\times 10^{-2}$ & 2.92$\times 10^{-2}$ & \textbf{4.11$\times 10^{-5}$} & 68.38 & \textbf{1.95} \bigstrut[t]\\
    Bank Marketing & \textbf{2.10$\times 10^{-3}$} & 6.28$\times 10^{-2}$ & 1.94$\times 10^{-01}$ & \textbf{7.93$\times 10^{-5}$} & 188.25 & \textbf{6.77} \\
    Bejing Pm25 & \textbf{8.59$\times 10^{-3}$} & 8.74$\times 10^{-3}$ & 8.18$\times 10^{-4}$ & \textbf{1.77$\times 10^{-5}$} & \textbf{0.21} & 5.08 \\
    Bikes Sharing & \textbf{2.68$\times 10^{-3}$} & 3.51$\times 10^{-2}$ & 3.34$\times 10^{-2}$ & \textbf{6.89$\times 10^{-5}$} & 158.46 & \textbf{2.65} \\
    California & \textbf{1.88$\times 10^{-2}$} & 5.50$\times 10^{-2}$ & 1.31$\times 10^{-01}$ & \textbf{1.35$\times 10^{-4}$} & 75.67 & \textbf{2.65} \\
    CCPP  & \textbf{4.89$\times 10^{-3}$} & 5.16$\times 10^{-2}$ & 1.61$\times 10^{-2}$ & \textbf{1.12$\times 10^{-4}$} & 17.05 & \textbf{1.16} \\
    Mv    & \textbf{3.20$\times 10^{-3}$} & 5.53$\times 10^{-2}$ & 1.76$\times 10^{-01}$ & \textbf{1.16$\times 10^{-4}$} & 167.36 & \textbf{4.85} \bigstrut[b]\\
    \hline
     \end{tabular}%

\end{adjustbox}

\label{tab:results_1x50}

\end{table}%

In order to a get a quick look at the results we also produce histograms in which we analyze the performance over each problem by changing the architectures. In this way it is possible to visualize how the performance over the same problem changes with the architecture of the network. B$^2$LD's performance seems to be not affected by the number of layers, whereas for deeper network LBFGS seems to be trapped into "bad" solutions.

\begin{figure}[tbp]
     \begin{center}
        \subfigure[Ailerons]{%
            \label{fig:first}
            \includegraphics[width=0.4\textwidth]{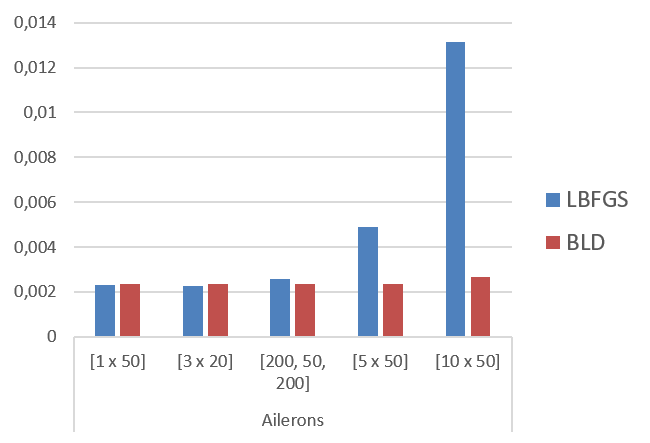}
        }%
        \subfigure[Bank Marketing ]{%
           \label{fig:second}
           \includegraphics[width=0.4\textwidth]{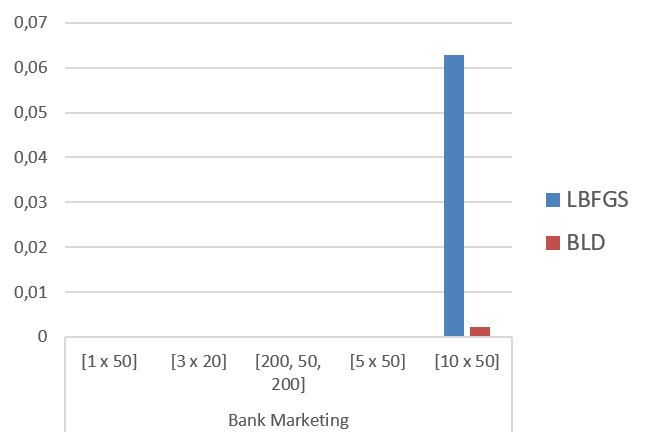}
        }\\ 
        \subfigure[Bejing PM 2.5]{%
            \label{fig:third1}
            \includegraphics[width=0.4\textwidth]{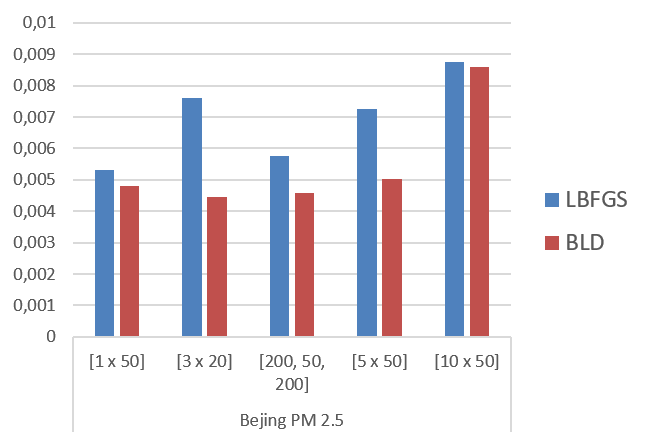}
        }%
        \subfigure[ Bikes Sharing]{%
            \label{fig:fourth1}
             \includegraphics[width=0.4\textwidth]{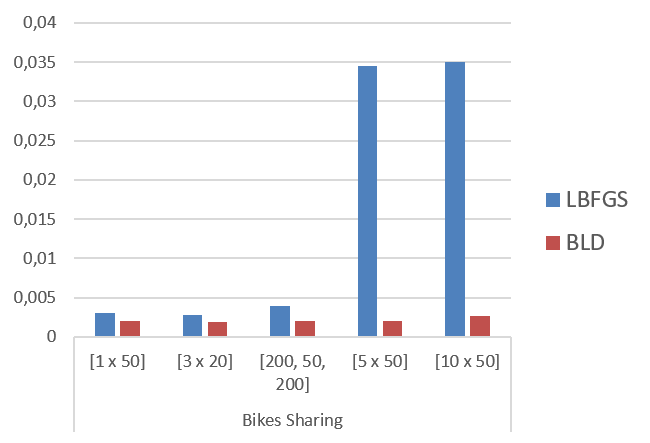}
        }\\
        
        \subfigure[ California ]{%
            \label{fig:third}
            \includegraphics[width=0.4\textwidth]{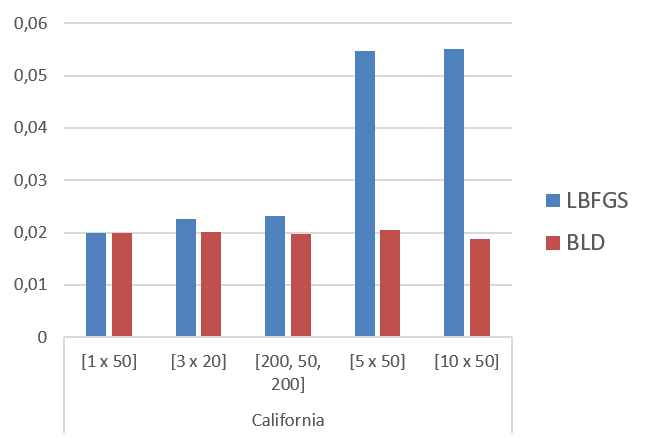}
        }%
        \subfigure[CCPP ]{%
            \label{fig:fourth2}
            \includegraphics[width=0.4\textwidth]{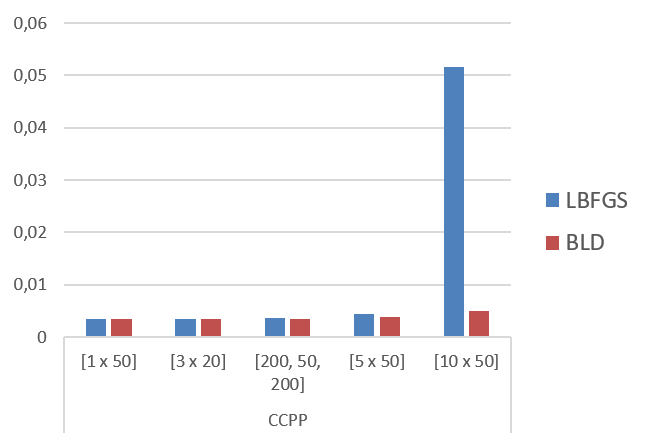}
        }\\
        \subfigure[Mv]{%
            \label{fig:fourth}
            \includegraphics[width=0.4\textwidth]{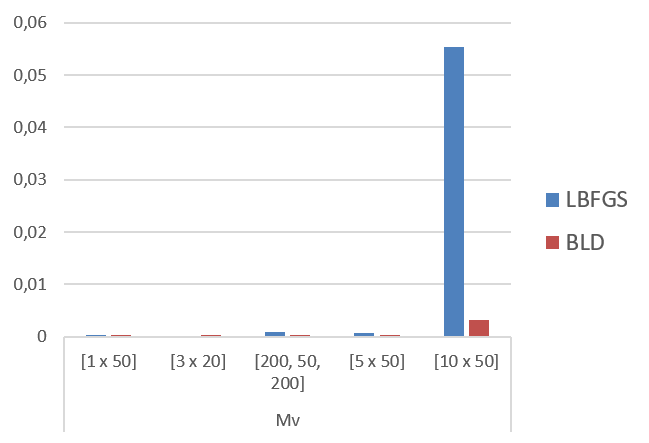}
        }
    \end{center}
    \caption{%
        Best Objective function value returned by the BLD  and the  LBFGS methods.
     }%
   \label{fig:subfigures}
\end{figure}

The average behaviour of B$^2$LD is assessed by counting the number of win, \#wins, (B$^2$LD beats LBFGS) and defeats, \#defs, (LBFGS beats B$^2$LD) over the 10 random runs. We say that a method beats the other if the returned value is better than 5\% compared with the other; otherwise we declare a tie. The cumulative results are reported in Table \ref{tab:win_lose_train} and \ref{tab:win_lose_test} respectively on the objective function and the test error value. 
For each problem and for each network architecture we report into parenthesis [\#wins ; \#defs] of B$^2$LD w.r.t. LBFG. From this analysis it is evident that B$^2$LD remarkably outperforms LBFGS particularly when the dimension of the problem increases; indeed it obtains better solutions (smaller objective function, Table \ref{tab:win_lose_train}) with better generalization properties (smaller test error, Table \ref{tab:win_lose_test}).

\begin{table}[htbp ]
  \centering
  \caption{[\#wins ; \#defs] of B$^2$LD versus LBFGS on the objective function values (10 runs)}
  \begin{adjustbox}{max width=\textwidth}
    \begin{tabular}{l|lcccc}
          & [1 x 50 ] & [3 x 20 ] & [200, 50, 200 ] & [5 x 50 ] & [10 x 50 ] \bigstrut[b]\\
    \hline
    Ailerons & [\, 0 ; 0 ] & [\, 6 ; 2 ] & [\, 8 ; 0 ] & [10 ; 0 ] & [\, 9 ; 0  ] \bigstrut[t ]\\
    Bank Marketing & [0 ; 10 ] & [\, 5 ; 5 ] & [\, 4 ; 5 ] & [\, 8 ; 2 ] & [10 ; 0 ] \\
    Bejing Pm25 & [\, 5 ; 0 ] & [10 ; 0 ] & [10 ; 0 ] & [10 ; 0 ] & [\, 0 ; 0 ] \\
    Bikes Sharing & [10 ; 0 ] & [10 ; 0 ] & [10 ; 0 ] & [10 ; 0 ] & [10 ; 0 ] \\
    California & [\, 8 ; 0 ] & [10 ; 0 ] & [10 ; 0 ] & [10 ; 0 ] & [10 ; 0 ] \\
    CCPP  & [\, 2 ; 0 ] & [\, 2 ; 6 ] & [\, 8 ; 0 ] & [10 ; 0 ] & [10 ; 0 ] \\
    Mv    & [\, 3 ; 5 ] & [\, 5 ; 3 ] & [10 ; 0 ] & [10 ; 0 ] & [10 ; 0 ] \bigstrut[b]\\
    \hline
    \end{tabular}%
    \end{adjustbox}
  \label{tab:win_lose_train}%
  
\end{table}%

\begin{table}[htbp ]
  \centering
  \caption{[\#wins ; \#defs] of B$^2$LD versus LBFGS 
on the test error values (10 runs)}
  \begin{adjustbox}{max width=\textwidth}
    \begin{tabular}{l|ccccc}
          & [1 x 50 ] & [3 x 20 ] & [200, 50, 200 ] & [5 x 50 ] & [10 x 50 ] \bigstrut[b]\\
    \hline
    Ailerons & [\, 0 ; 0 ] & [\, 6 ; 4 ] & [\, 7 ; 0 ] & [10 ; 0 ] & [\, 9 ; 0 ] \bigstrut[t ]\\
    Bank Marketing & [\, 0 ; 9 ] & [\, 5 ; 5 ] & [\, 2 ; 8 ] & [\, 8 ; 2 ] & [10 ; 0 ] \\
    Bejing Pm25 & [\, 5 ; 0 ] & [10 ; 0 ] & [10 ; 0 ] & [10 ; 0 ] & [\, 0 ; 0 ] \\
    Bikes Sharing & [10 ; 0 ] & [10 ; 0 ] & [10 ; 0 ] & [10 ; 0 ] & [10 ; 0 ] \\
    California & [\, 8 ; 0 ] & [\, 9 ; 0 ] & [10 ; 0 ] & [10 ; 0 ] & [10 ; 0 ] \\
    CCPP  & [\, 0 ; 0 ] & [\, 1 ; 8 ] & [\, 8 ; 0 ] & [\, 9 ; 1 ] & [10 ; 0 ] \\
    Mv    & [\, 3 ; 6 ] & [\, 4 ; 6 ] & [10 ; 0 ] & [\, 8 ; 1 ] & [10 ; 0 ] \bigstrut[b]\\
    \hline
    \end{tabular}%
    \end{adjustbox}
    
  \label{tab:win_lose_test}%
\end{table}%


Finally, we analyzed how many times each layer is updated when using B$^2$LD. Indeed, while  LBFGS method updates the weights of all the layers at the same time, B$^2$LD considers layers one by one and it is possible that some layer is updated less often than others.  Indeed, whenever the gradient with respect to a layer and/or the relative reduction of the objective function are below a given tolerance, B$^2$LD skips the update of the layer. In this way, the algorithm does not waste time in optimizing parts of the network not worthwhile. We consider the deepest network, [10 x 50],  and we count the number of updates performed by the best BLD on each layer, which are reported in Table \ref{tab:strati}.
  It turns out that the first and last layers are optimized much often than layers in the middle of the network. This highlight the fact that the network might be too deep for the problems and a smaller number of hidden layers may be enough.  We also note that on the dataset Bejing PM 2.5, B$^2$LD  performs only the optimization of the last layer, miming the behaviour of Extreme Learning Machine and it finds a better solution than LBFGS (see Table \ref{tab:results_1x50}).

\begin{table}[htbp]
  \centering
  \caption{Number of updates per layer $\ell=0,\dots, 10$ performed by B$^2$LD  during the best run  for the network [10 x 50].}
  \begin{adjustbox}{max width=\textwidth}
    \begin{tabular}{l|ccccccccccc}
          & \multicolumn{11}{c}{Layer of the network} \\
    Dataset  & 0 & 1     & 2     & 3     & 4     & 5     & 6     & 7     & 8     & 9     & 10 \bigstrut[b]\\
    \hline
    Mv    & 30    & 30    & 48    & 8     & 2     & 2     & 2     & 2     & 2     & 2     & 14 \bigstrut[t]\\
    Bikes Sharing & 120   & 62    & 35    & 6     & 62    & 5     & 4     & 4     & 4     & 4     & 18 \\
    Bejing PM 2.5 & 0     & 0     & 0     & 0     & 0     & 0     & 0     & 0     & 0     & 0     & 5 \\
    CCPP  & 11    & 2     & 2     & 31    & 2     & 2     & 2     & 2     & 2     & 2     & 4 \\
    Ailerons & 33    & 16    & 24    & 63    & 4     & 7     & 4     & 4     & 4     & 4     & 21 \\
    California & 28    & 28    & 39    & 11    & 2     & 2     & 2     & 2     & 2     & 2     & 13 \\
    Bank Marketing & 30    & 30    & 30    & 30    & 1     & 6     & 2     & 30    & 1     & 1     & 5 \bigstrut[b]\\
    \hline
    \end{tabular}%
    \end{adjustbox}
    
  \label{tab:strati}%
\end{table}%

\section{{Minibatch} Block Layer Decomposition algorithm}\label{Minibatch BCD}

So far, we have seen that the BLD method can improve the performance of standard optimization algorithms.
In this section, we exploited the addictive structure of the objective function and we embed the layer decomposition scheme proposed above into a \textit{minibatch} strategy with the aim of enhancing the performance of both online methods and batch decomposition methods.
A similar idea has been proposed in \cite{bravi2014incremental} where a two-block incremental decomposition method has been proposed which is suitable for block layer decomposition of a shallow network.

\subsection{Minibatch Block Decomposition Method}
Let ${\cal B}_h\subset \{1,\dots, P\}$ be the index set identifying a minibatch, i.e. a subset of the samples. Given a partition ${\cal B}=\lbrace {\cal B}_1,\dots,{\cal B}_H\rbrace$  of the set $\{1,\dots, P\}$, problem \eqref{optimization_problem} can be expressed as
$$\min_w \sum_{h=1}^H \sum_{p\in {\cal B}_h} E_p(w_1,\dots, w_L) +\rho \|w\|^2$$
}

The \textit{minibatch}  BCD algorithm applies a double decomposition to the problem, meaning that at each iteration only information on the function involving a minibatch ${\cal B}_h$ is used to update a subset of variables, the \textit{working set}, $ {\cal J}^k$.


Let us introduce the following notation
$$f_h=\sum_{p\in {\cal B}_h} E_p(w_1,\dots, w_L) +|{\cal B}_h|\frac{\rho}{P} \|w\|^2$$ so that 
$\displaystyle f(w)=\sum_{h=1}^H f_h$ and the gradient can be expressed as 
{$$\nabla f(w)=
\sum_{h=1}^H  \sum_{\ell=1}^L e_\ell \otimes \nabla_{w_\ell} f_h
$$
where $\otimes$ represents the standard Kronecker Product and $e_\ell\in \R^L$ is  the $\ell$-th column of the identity matrix of dimension $L$.  
}


A very general framework of a minibatch block layer decomposition is reported in Algorithm \ref{alg:minibatch_BCD}.

\begin{algorithm}
    \caption{{Minibatch} Block Layer Decomposition}
    \label{alg:minibatch_BCD}
    \begin{algorithmic}[1]
\State Given $\lbrace x_p,y_p\rbrace_{p=1}^P 
$
\State Choose $w^0\in\rm I\!R^n$ and set  $k=0$; 
\While {(stopping criterion not met)}
\State Define a partition ${\cal B}^k=\lbrace {\cal B}_1,\dots,{\cal B}_H\rbrace$ of $\{1,\dots, P\}$
\State Select a minibatch ${\cal B}_h $, $h\in\{1,\dots,H\}$
\State Choose  ${\cal J}_h^k\subseteq \{1,\dots, L\}$
\State Set $\tilde w^0 =w^k$; $j=0$;
\For {$\ell\in{\cal J}_h^k$}
\State Set $d_\ell^j=-\nabla_{w_\ell} f_h(\tilde w^j)$
\State Update $\widetilde w_\ell^{j+1}=\tilde w_\ell^j+\alpha^kd_\ell^j$ 
\State $j=j+1$
\EndFor
\State $w^{k+1}=\tilde w^{|{\cal J}_h^k|}$
\State Update $\alpha^k$ 
\State $k=k+1$
\EndWhile
\end{algorithmic}
\end{algorithm}

Algorithm \ref{alg:minibatch_BCD} can be differently characterized depending on the definition of the partition ${\cal B}^k$, the selection rule of the minibatch ${\cal B}_h$, the choice of the working set ${\cal J}_h^k$, and the updating rule of the stepsize $\alpha^k$.
We address possible interesting choices in the following paragraphs. 
\subsection*{Selection of the minibatch ${\cal B}_h$}
Selection of the minibatch over the $P$ samples can be implemented following classical rules used in \textit{online} methods for Machine Learning.
In particular we can consider
\begin{itemize}
\item \textbf{Incremental Rule}: the order in which minibatches will be used is fixed \textit{a priori} and kept unchanged over the iterations;

\item \textbf{Stochastic Rule}: at each iteration a minibatch is chosen randomly from the available list $\cal B$;

\item \textbf{Random without replacement rule}: at each iteration a minibatch is chosen randomly in $\cal B$ without replacement; this corresponds to an incremental rule when the selection order is reshuffled;

\end{itemize}

\subsection*{Selection of the \textit{working set} ${\cal J}_h^k$}
The index set ${\cal J}_h^k$ defines which blocks will be sequentially updated at iteration $k$ using the minibatch ${\cal B}_h$. ${\cal J}_h^k$ can be composed by all the layers, ${\cal J}_h^k=\lbrace 1,...,L\rbrace$, or can be a subset of the index set  ${\cal L}$. In the first case, all the blocks are updated sequentially using the same minibatch, while, in the second case, only a subset of the layers are updated with a given minibatch. Furthermore, ${\cal J}_h^k$ can be defined at each iteration, or it can be fixed \textit{a priori}, ${\cal J}_h^k={\cal J}$ $\forall k$ .

In the definition of the set ${\cal J}^k$, a computationally efficient choice, that exploits the backpropagation rule for updating the gradient, consists in setting set ${\cal J}^k=\{1,\dots,L\}$. Indeed, when updating a block of variables $w_{\ell^k}$ with a minibatch ${\cal B}_h$ in each layer $\ell$  we can store the output $z_{\ell}^i$ of each neuron $i$  so that if the same minibatch is used to update another block of variables $w_{\ell^{k+1}}$ we do not need to evaluate the output $z_\ell$ of the previous layers $\ell<\ell^{k+1}$, being these unchanged.

\subsection{Implementation and performance of a {minibatch} BLD method}
For the experimental results, we implemented a \textit{minibatch} BCD algorithm with the following settings:
\begin{itemize}
\item the order in which minibatches ${\cal B}_h$ are used is fixed \textit{a priori};
\item  at each iteration the set ${\cal J}^k_h$ is composed by all the layers which are updated sequentially following a \textit{backward} order as in the B$^2$LD scheme presented above;
\item $\alpha^k$ updated according to the diminishing stepsize $$\alpha^k(1-\epsilon\alpha^k)$$
\end{itemize}
At each iteration, the algorithm picks a minibatch ${\cal B}_h$ and updates the layers of the network sequentially in backward order by performing a steepest descent iteration. After having updated all the layers, a new minibatch is considered and the same procedure is applied until a certain stopping criterion, such as a maximum number of iterations or the norm of the gradient smaller than a certain threshold, is met.
We call this scheme Block Layer Incremental Gradient (BLInG) which is reported in Algorithm \ref{alg:minibatch_Bling}.

\begin{algorithm}
    \caption{Block Layer Incremental Gradient (BLInG)}
    \label{alg:minibatch_Bling}
    \begin{algorithmic}[1]
\State Given $\lbrace x_p,y_p\rbrace_{p=1}^P , {\cal L}=\lbrace1,...,L\rbrace 
$
\State Choose $w^0\in\rm I\!R^n$ and set  $k=0,\alpha^0\geq 0,\epsilon\in(0;1),\beta>0, \gamma<\infty,$; 
\State Define a partition ${\cal B}=\lbrace {\cal B}_1,\dots,{\cal B}_H\rbrace$ of $\{1,\dots,P\}$
\While {(stopping criterion not met)}
\For {h=1,...,H}
\State Set $\tilde w^0 =w^k$; $j=0$;
\For {$\ell=1,...,L$}
\State Set $d_\ell^j=-\nabla_{w_\ell} f_h(\tilde w^j)$
\State Update $\tilde w_\ell^{j+1}=\tilde w_\ell^j+\frac{\alpha^k}{\max\lbrace\beta,min\lbrace\gamma,\Vert d_\ell^k\Vert\rbrace\rbrace} d_\ell^j$ 
\State $j=j+1$
\EndFor
\State $w^{k+1}=\tilde w^{|{\cal J}_h^k|}$
\State $\alpha^{k+1}=\alpha^k(1-\epsilon\alpha^k)$
\State $k=k+1$
\EndFor
\EndWhile
\end{algorithmic}
\end{algorithm}

Convergence of the BLInG above can be proved under suitable assumptions following \cite{Bertsekas2000} by looking at the iteration generated by BLInG as a gradient method with error. A similar approach has been followed in \cite{bravi2014incremental} for a two-block decomposition where the last block is optimized with a full batch strategy.

Following \cite{Yu2017a}, in the updating formula of $w_\ell^k$ we have scaled the stepsize with a normalization term, properly bounded to avoid overflow, since this choice seemed to make the updating process more robust by avoiding vanishing and exploding gradient. 

As regards the updating rule for the stepsize $\alpha^k$, convergence requires to use a \textit{diminishing} stepsize rule.
The proposed rule at Step 13 of the BLInG algorithm satisfies the assumptions as well as another usual updating rule $\alpha^{k+1}=\frac{1}{k}\alpha^k$.
Concerning the parameters in the updating rule of the stepsize, $\gamma$ and $\beta$ were fixed to $10^{-3}$ and $10^6$ respectively. 


We compared BLInG with the Incremental Gradient (IG), which is its non-decomposed counterpart.
We use for both the two algorithms the same setting and normalization for the stepsize $\alpha^k$.
The initial values of the learning rate $\alpha^0$ and the fraction reduction $\epsilon$ were chosen through a grid-search procedure which led to different values for the two methods which depend on the network architecture: 
$$\alpha^0_{IG}=0.5 \qquad \alpha^0_{BLInG}=\frac{0.5}{\max \lbrace 1, L-2\rbrace} \qquad \epsilon=5 \times 10^{-3}$$
IG performance was invariant with respect to the initial value of the learning rate which best value was the same for all the architectures of the network; BLInG instead performs better with a  smaller initial learning rate for deeper networks.

Datasets and architectures of the networks are the same of those used for the BLD algorithm and reported in Table \ref{tab:dataset}. Each algorithm was tested over 10 runs starting from randomly chosen initial points.

Since \textit{minibatch} methods evaluate neither the objective function nor the gradient at each iteration, the stopping criterion 
has been a limit on the computational time that was fixed to 60 seconds.

\begin{table}[htbp]
  \centering  
    \caption{
    Best results returned over 10 runs by the BLInG and IG}
  \begin{adjustbox}{max width=\textwidth}
    \begin{tabular}{l|rr|rr}
    \multicolumn{1}{p{4.215em}|}{Network } & \multicolumn{2}{c|}{Objective Value} & \multicolumn{2}{c}{Test Error} \bigstrut[b]\\
\cline{2-5}     [1 x 50] & \multicolumn{1}{c}{BLInG} & \multicolumn{1}{c|}{IG} & \multicolumn{1}{c}{BLInG} & \multicolumn{1}{c}{IG} \bigstrut\\
    \hline
    Ailerons & 2,01$\times 10^{-3}$ & 1,99$\times 10^{-3}$ & 1,98$\times 10^{-3}$ & 1,93$\times 10^{-3}$ \bigstrut[t]\\
    Bank Marketing & 2,45$\times 10^{-6}$ & 2,01$\times 10^{-6}$ & 2,83$\times 10^{-5}$ & 1,29$\times 10^{-5}$ \\
    Bejing PM 2.5 & 4,26$\times 10^{-3}$ & 4,63$\times 10^{-3}$ & 3,98$\times 10^{-3}$ & 4,28$\times 10^{-3}$ \\
    Bikes Sharing & 1,79$\times 10^{-3}$ & 2,21$\times 10^{-3}$ & 1,95$\times 10^{-3}$ & 2,38$\times 10^{-3}$ \\
    California & 1,69$\times 10^{-2}$ & 1,89$\times 10^{-2}$ & 1,74$\times 10^{-2}$ & 1,93$\times 10^{-2}$ \\
    CCPP  & 3,17$\times 10^{-3}$ & 3,21$\times 10^{-3}$ & 3,14$\times 10^{-3}$ & 3,18$\times 10^{-3}$ \\
    Mv    & 7,40$\times 10^{-5}$ & 8,03$\times 10^{-5}$ & 7,72$\times 10^{-5}$ & 8,22$\times 10^{-5}$ \bigstrut[b]\\
    \hline
    \hline
\cline{2-5}     [3 x 20] & \multicolumn{1}{c}{BLInG} & \multicolumn{1}{c|}{IG} & \multicolumn{1}{c}{BLInG} & \multicolumn{1}{c}{IG} \bigstrut\\
    \hline
    Ailerons & 2,02$\times 10^{-3}$ & 2,05$\times 10^{-3}$ & 1,94$\times 10^{-3}$ & 2,00$\times 10^{-3}$ \bigstrut[t]\\
    Bank Marketing & 2,40$\times 10^{-7}$ & 1,26$\times 10^{-6}$ & 1,07$\times 10^{-6}$ & 4,24$\times 10^{-6}$ \\
    Bejing PM 2.5 & 4,12$\times 10^{-3}$ & 4,19$\times 10^{-3}$ & 3,83$\times 10^{-3}$ & 3,94$\times 10^{-3}$ \\
    Bikes Sharing & 1,78$\times 10^{-3}$ & 1,80$\times 10^{-3}$ & 1,87$\times 10^{-3}$ & 1,93$\times 10^{-3}$ \\
    California & 1,61$\times 10^{-2}$ & 1,65$\times 10^{-2}$ & 1,68$\times 10^{-2}$ & 1,74$\times 10^{-2}$ \\
    CCPP  & 3,14$\times 10^{-3}$ & 3,13$\times 10^{-3}$ & 3,07$\times 10^{-3}$ & 3,07$\times 10^{-3}$ \\
    Mv    & 1,46$\times 10^{-5}$ & 2,10$\times 10^{-5}$ & 1,51$\times 10^{-5}$ & 2,15$\times 10^{-5}$ \bigstrut[b]\\
    \hline
    \hline
\cline{2-5}     [200,50,200] & \multicolumn{1}{c}{BLInG} & \multicolumn{1}{c|}{IG} & \multicolumn{1}{c}{BLInG} & \multicolumn{1}{c}{IG} \bigstrut\\
    \hline
    Ailerons & 2,09$\times 10^{-3}$ & 2,14$\times 10^{-3}$ & 2,03$\times 10^{-3}$ & 2,07$\times 10^{-3}$ \bigstrut[t]\\
    Bank Marketing & 7,14$\times 10^{-6}$ & 2,16$\times 10^{-4}$ & 1,59$\times 10^{-5}$ & 3,22$\times 10^{-4}$ \\
    Bejing PM 2.5 & 4,60$\times 10^{-3}$ & 5,02$\times 10^{-3}$ & 4,27$\times 10^{-3}$ & 4,68$\times 10^{-3}$ \\
    Bikes Sharing & 2,52$\times 10^{-3}$ & 3,10$\times 10^{-3}$ & 2,53$\times 10^{-3}$ & 3,11$\times 10^{-3}$ \\
    California & 1,81$\times 10^{-2}$ & 2,15$\times 10^{-2}$ & 1,90$\times 10^{-2}$ & 2,22$\times 10^{-2}$ \\
    CCPP  & 3,16$\times 10^{-3}$ & 3,22$\times 10^{-3}$ & 3,10$\times 10^{-3}$ & 3,17$\times 10^{-3}$ \\
    Mv    & 4,87$\times 10^{-5}$ & 5,40$\times 10^{-4}$ & 4,90$\times 10^{-5}$ & 5,33$\times 10^{-4}$ \bigstrut[b]\\
    \hline
   \hline
\cline{2-5}     [5 x 50] & \multicolumn{1}{c}{BLInG} & \multicolumn{1}{c|}{IG} & \multicolumn{1}{c}{BLInG} & \multicolumn{1}{c}{IG} \bigstrut\\
    \hline
    Ailerons & 2,04$\times 10^{-3}$ & 2,10$\times 10^{-3}$ & 1,99$\times 10^{-3}$ & 2,04$\times 10^{-3}$ \bigstrut[t]\\
    Bank Marketing & 6,99$\times 10^{-7}$ & 1,18$\times 10^{-5}$ & 2,46$\times 10^{-6}$ & 2,13$\times 10^{-5}$ \\
    Bejing PM 2.5 & 4,22$\times 10^{-3}$ & 4,35$\times 10^{-3}$ & 3,92$\times 10^{-3}$ & 4,03$\times 10^{-3}$ \\
    Bikes Sharing & 2,02$\times 10^{-3}$ & 2,19$\times 10^{-3}$ & 2,15$\times 10^{-3}$ & 2,32$\times 10^{-3}$ \\
    California & 1,68$\times 10^{-2}$ & 1,77$\times 10^{-2}$ & 1,73$\times 10^{-2}$ & 1,83$\times 10^{-2}$ \\
    CCPP  & 3,10$\times 10^{-3}$ & 3,15$\times 10^{-3}$ & 3,04$\times 10^{-3}$ & 3,09$\times 10^{-3}$ \\
    Mv    & 1,14$\times 10^{-5}$ & 5,01$\times 10^{-5}$ & 1,14$\times 10^{-5}$ & 5,07$\times 10^{-5}$ \bigstrut[b]\\
    \hline
    \hline
\cline{2-5}     [10 x 50] & \multicolumn{1}{c}{BLInG} & \multicolumn{1}{c|}{IG} & \multicolumn{1}{c}{BLInG} & \multicolumn{1}{c}{IG} \bigstrut\\
    \hline
    Ailerons & 2,13$\times 10^{-3}$ & 2,17$\times 10^{-3}$ & 2,12$\times 10^{-3}$ & 2,12$\times 10^{-3}$ \bigstrut[t]\\
    Bank Marketing & 8,32$\times 10^{-7}$ & 5,01$\times 10^{-5}$ & 7,45$\times 10^{-7}$ & 4,19$\times 10^{-5}$ \\
    Bejing PM 2.5 & 4,75$\times 10^{-3}$ & 8,68$\times 10^{-3}$ & 4,45$\times 10^{-3}$ & 8,15$\times 10^{-3}$ \\
    Bikes Sharing & 2,88$\times 10^{-3}$ & 4,16$\times 10^{-3}$ & 2,95$\times 10^{-3}$ & 4,31$\times 10^{-3}$ \\
    California & 1,80$\times 10^{-2}$ & 5,49$\times 10^{-2}$ & 1,88$\times 10^{-2}$ & 5,81$\times 10^{-2}$ \\
    CCPP  & 3,30$\times 10^{-3}$ & 3,49$\times 10^{-3}$ & 3,21$\times 10^{-3}$ & 3,40$\times 10^{-3}$ \\
    Mv    & 3,58$\times 10^{-4}$ & 8,31$\times 10^{-4}$ & 3,51$\times 10^{-4}$ & 7,96$\times 10^{-4}$ \bigstrut[b]\\
    \hline
    \end{tabular} \end{adjustbox}
  \label{tab:m_results} 
\end{table}%

The results are reported in
Tables \ref{tab:m_results} and \ref{tab:mini_win_lose_test}. In Table \ref{tab:m_results} for each of the two algorithms, we report the best values of the objective function found over the 10 runs and the corresponding value of the test error within the time limit of 60 sec.
Overall, BLInG is able to return better solutions with better generalization properties on the test set than those returned by IG, especially when the dimension of the network blows up.  Table \ref{tab:mini_win_lose_test}, provides cumulative information on the 10 runs. In particular, we considered the number of wins or defeats of the BLInG algorithm versus IG over the 10 runs (being a tie a result within the 5\%). Similarly to the BLD algorithm, BLInG seems to perform better when the dimension of the problem increases. In order to assess how much the depth of the network influences algorithms' performance, in Table \ref{tab:ratio} the ratio between best value found in the network [10x50] and [1x50] are provided. Higher values mean that the algorithm has been harmed by the increased depth of the network. Comparing these solutions, BLInG turns out to be less affected by the increased structure of the network and is able to find always similar values regardless of the structure of the network while IG performs worse.

\begin{table}[htbp ]
  \centering
  \caption{Cumulative comparison [\#wins ; \#defs] of BLInG and IG algorithms on training and test error value over the 10 random runs }
  \begin{adjustbox}{max width=\textwidth}
    \begin{tabular}{l|ccccc}
          & \multicolumn{5}{c}{Objective value [\#wins ;  \#defs ]} \\
          & [1 x 50 ] & [3 x 20 ] & [200,  50,  200 ] & [5 x 50 ] & [10 x 50 ] \bigstrut[b ]\\
    \hline
    Ailerons & [\, 0 ; 1 ] & [\, 0 ; 0 ] & [\, 2 ; 3 ] & [\, 1 ; 4 ] & [\, 9 ; 1 ] \bigstrut[t ]\\
    Bank Marketing & [\, 3 ; 7 ] & [\, 4 ; 6 ] & [\, 7 ; 3 ] & [\, 9 ; 1 ] & [10 ; 0 ] \\
    Bejing PM 2.5 & [10 ; 0 ] & [\, 0 ; 1 ] & [\, 4 ; 5 ] & [\, 1 ; 2 ] & [\, 6 ; 0 ] \\
    Bikes Sharing & [10 ; 0 ] & [\, 0 ; 4 ] & [\, 8 ; 2 ] & [\, 4 ; 2 ] & [10 ; 0 ] \\
    California & [\, 8 ; 0 ] & [\, 0 ; 0 ] & [\, 7 ; 3 ] & [\, 1 ; 1 ] & [10 ; 0 ] \\
    CCPP  & [\, 0 ; 0 ] & [\, 0 ; 1 ] & [\, 3 ; 6 ] & [\, 0 ; 2 ] & [\, 5 ; 0 ] \\
    Mv    & [\, 5 ; 1 ] & [\, 5 ; 5 ] & [\, 9 ; 0 ] & [10 ; 0 ] & [10 ; 0 ] \bigstrut[b ]\\
    \hline
    \hline
          & \multicolumn{5}{c}{Test [\#wins ;  \#defs ]} \\
          & [1 x 50 ] & [3 x 20 ] & [200,  50,  200 ] & [5 x 50 ] & [10 x 50 ] \bigstrut[b ]\\
    \hline
    Ailerons & [\, 0 ; 3\ ] & [\, 0 ; 0 ] & [\, 2 ; 3 ] & [\, 1 ; 4 ] & [\, 9 ; 1 ] \bigstrut[t ]\\
    Bank Marketing & [\, 3 ; 7 ] & [\, 3 ; 7 ] & [\, 8 ; 2 ] & [10 ; 0 ] & [10 ; 0 ] \\
    Bejing PM 2.5 & [10 ; 0 ] & [\, 0 ; 0 ] & [\, 4 ; 5 ] & [\, 1 ; 2 ] & [\, 6 ; 0 ] \\
    Bikes Sharing & [10 ; 0 ] & [\, 0 ; 3 ] & [\, 8 ; 2 ] & [\, 3 ; 2 ] & [10 ; 0 ] \\
    California & [\, 6 ; 0 ] & [\, 0 ; 0 ] & [\, 7 ; 3 ] & [\, 1 ; 1 ] & [10 ; 0 ] \\
    CCPP  & [\, 0 ; 0 ] & [\, 0 ; 1 ] & [\, 4 ; 6 ] & [\, 2 ; 2 ] & [\, 4 ; 0 ] \\
    Mv    & [\, 5 ; 1 ] & [\, 5 ; 5 ] & [\, 9 ; 0 ] & [10 ; 0 ] & [10 ; 0 ] \bigstrut[b ]\\
    \hline
    \end{tabular}%
  \end{adjustbox}\label{tab:mini_win_lose_test}%
\end{table}%

\begin{table}[htbp]
  \centering
  \caption{Ratio between best value found in [10x50] and [1x50]}
    \begin{tabular}{l|cc|cc}
    \multicolumn{1}{p{4.215em}|}{Network } & \multicolumn{2}{c|}{Objective value} & \multicolumn{2}{c}{Test Error} \bigstrut[b]\\
\cline{2-5}    [1x50]/[10x50] & BLInG & IG    & BLInG & IG \bigstrut\\
    \hline
    Ailerons & 1,06  & 1,09  & 1,07  & 1,10 \bigstrut[t]\\
    Bank Marketing & 0,34  & 24,91 & 0,03  & 3,25 \\
    Bejing PM 2.5 & 1,12  & 1,88  & 1,12  & 1,91 \\
    Bikes Sharing & 1,61  & 1,89  & 1,51  & 1,81 \\
    California & 1,06  & 2,91  & 1,08  & 3,01 \\
    CCPP  & 1,04  & 1,09  & 1,02  & 1,07 \\
    Mv    & 4,84  & 10,35 & 4,54  & 9,68 \bigstrut[b]\\
    \hline
    \end{tabular}%
  \label{tab:ratio}%
\end{table}%

\begin{figure}[tbp]
     \begin{center}
        \subfigure[Ailerons ]{%
            \label{fig:first2}
            \includegraphics[width=0.4\textwidth]{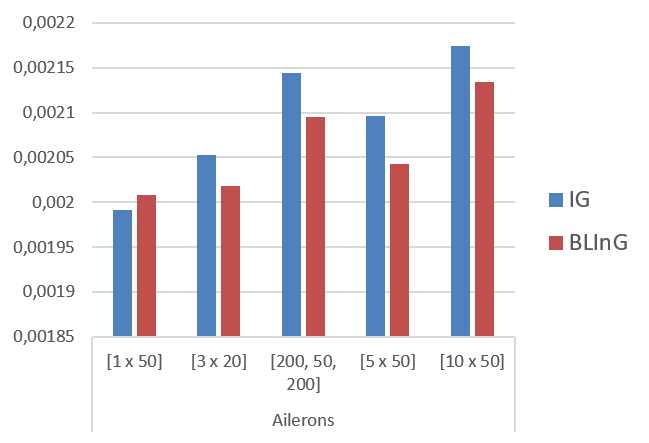}
        }%
        \subfigure[Bank Marketing ]{%
           \label{fig:second2}
           \includegraphics[width=0.4\textwidth]{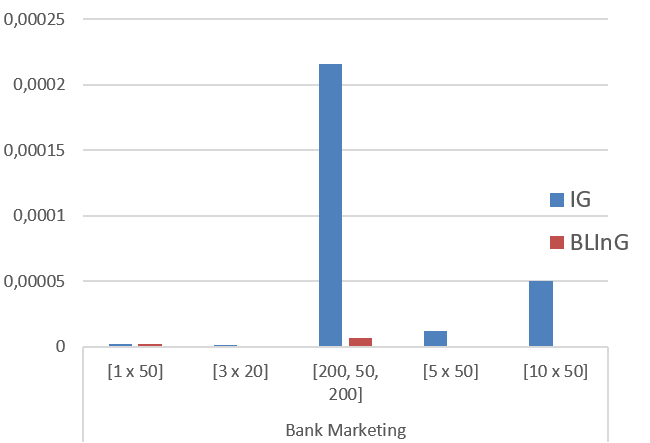}
        }\\ 
        \subfigure[Bejing PM 2.5]{%
            \label{fig:third2}      
            \includegraphics[width=0.4\textwidth]{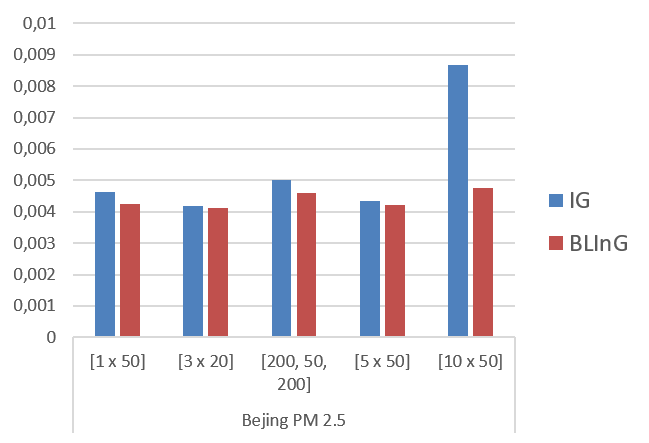}
        }%
        \subfigure[ Bikes Sharing]{%
            \label{fig:fourth7}
            \includegraphics[width=0.4\textwidth]{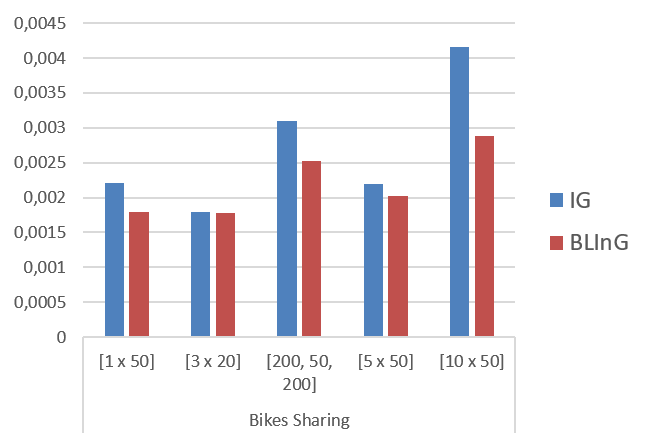}
        }\\
        
        \subfigure[ California ]{%
            \label{fig:third3}
            \includegraphics[width=0.4\textwidth]{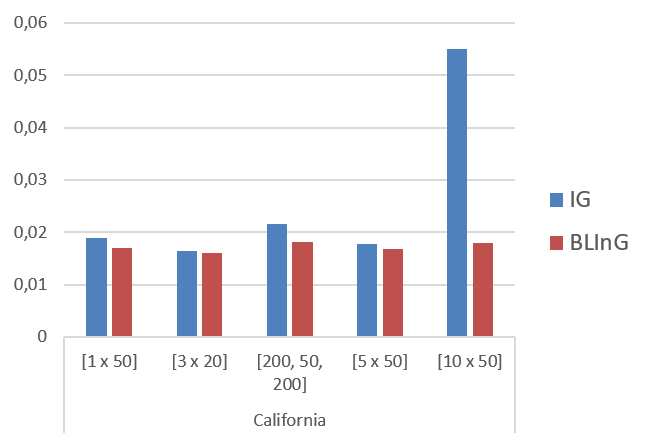}
        }%
        \subfigure[CCPP ]{%
            \label{fig:fourth3}
            \includegraphics[width=0.4\textwidth]{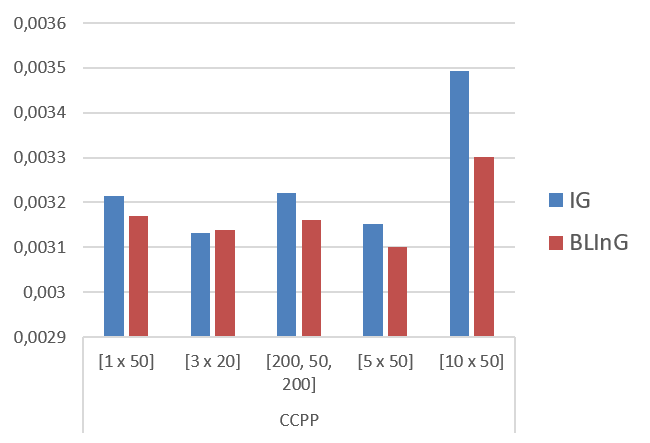}
        }\\
        \subfigure[Mv]{%
            \label{fig:fourth4}
            \includegraphics[width=0.4\textwidth]{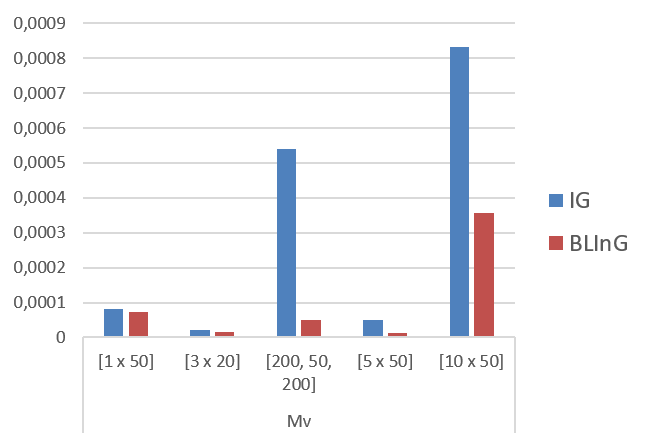}
        }
    \end{center}
    \caption{%
        Best training error value obtained by IG and BLInG within 60 seconds.
     }%
   \label{fig:subfigures1}
\end{figure}

\section{Conclusions}\label{Conclusions}
In this work, we focused on the application of batch and online Block Coordinate Decomposition methods for training Deep Feedforward Neural Networks. We studied how the layered structure of a DFNN can be effectively leveraged for training these models and we defined general  \textit{batch} and \textit{minibatch} block layer decomposition schemes.  Extensive numerical experiments over different network architectures have been performed to assess how  performance of state-of-the-art algorithms can be improved. Overall, the application of BCD methods turned out to be effective in avoiding bad attraction regions and speeding up the training process of DFNNs. Both the two proposed methods outperformed no-variable-decomposed counterparts leading to better solutions with better generalization properties as well.




\appendix
\section*{Appendix} \label{Appendix}
Before providing convergence proof of algorithm BLD, we need to recall some properties guaranteed by Armijo Linesearch whose proof can be found here \cite{Grippo2016}. 

\begin{lemma}\label{Lemma_Arm}
Armijo Linesearch determines a stepsize $\alpha^k\in(0,a]$ within a finite number of iterations. Furthermore, given a sequence $\lbrace w^k\rbrace$ with an accumulation point $\bar{w}$, if the following limit holds
$$\lim_{k\rightarrow \infty}f(w^k)-f(w_1^k,...,w_\ell^k+\alpha^kd^k_\ell,...,w_L)=0$$
then $\nabla_\ell f(\bar{w})=0$
\end{lemma}

First we show that an Armijo Linesearch coupled with a steepest descent method is always able to find a point $\tilde{w}$ such that conditions \ref{cond1} and \ref{cond2} are satisfied.
\begin{theorem}
Given a function $f:\rm I\!R^n \rightarrow [0,\infty)$, conditions (\ref{cond1}) and (\ref{cond2}) can be satisfied by the application of a finite number $n$ of iterations of the gradient method with an Armijo Linesearch.\\
\end{theorem}

\begin{proof}
(\ref{cond1}) is trivially satisfied $\forall k\geq 1$ since $\nabla f(w)^Td^k=-\Vert \nabla f(w)\Vert^2\leq 0$ $\forall k$. Concerning condition (\ref{cond2}), since $\alpha^k$ is chosen according to Armijo Linesearch, we have that
$$f(w^{k+1})\leq f(w^{k})-\gamma \alpha^k \Vert \nabla f(w^k)\Vert^2 \qquad \gamma\in (0,1)$$
Iterating this condition for $n$ steps we get:
\begin{equation} \label{eq: k_steps}
f(w^{k+n})\leq f(w^k)-\gamma \Big( \sum_{i=0}^{n-1} \alpha^{k+i}\Vert \nabla f(w^{k+i})\Vert^2\Big)
\end{equation}

Since $\alpha^k\nabla f(w^k)= w^{k}-w^{k+1}$ we have that
\begin{equation} \label{eq:update}
{\alpha^k}\Vert\nabla f(w^{k})\Vert^2=\frac{\Vert w^{k}-w^{k+1}\Vert^2}{\alpha^k}
\end{equation}
using (\ref{eq: k_steps}) and (\ref{eq:update}) we obtain:
\begin{equation}
f(w^{k+n})\leq f(w^{k})-\gamma \sum_{i=0}^{n-1}\frac{ \Vert w^{k+i}-w^{k+1+i}\Vert^2}{\alpha^{k+i}}
\end{equation}
By triangle inequality we have that:
$$\Vert w^{k+n}-w^{k}\Vert^2 \leq \sum_{i=0}^{n-1} \Vert w^{k+i}-w^{k+1+i}\Vert^2 
$$
namely:
$$- \sum_{i=0}^{n-1} \Vert w^{k+i}-w^{k+1+i}\Vert^2 \leq-\Vert w^{k+n}-w^{k}\Vert^2$$
which implies:
\begin{align*}
f(w^{k+n}) & \leq f(w^{k})-\gamma \frac{1}{\alpha^{k+i}}  \Vert w^{k+n}-w^{k}\Vert^2\\
& \leq f(w^{k})-\sigma\Big(  \Vert w^{k+n}-w^{k}\Vert^2\Big)
\end{align*}
The function $\sigma(t^k)=\gamma \frac{1}{\alpha^{k+i}} (t^k)^2$ is a forcing function because:
$$\lim_{k\rightarrow\infty}\gamma \frac{1}{\alpha^{k+i}} (t^k)^2=0 \implies \lim_{k\rightarrow\infty}(t^k)^2=0$$
indeed $\gamma>0$ and the stepsize $\alpha^k\in (0, a]$ $\forall k$.\\
\end{proof}

We provide the proof of Theorem \ref{th:convBLD} for the convergence of the  BLD method to stationary non minima points with a {\textit{forward}} cyclical selection of the blocks. This proof can be easily extended to other selections rules by tedious variations of the one provided below.

\begin{proof} {\bf of Theorem \ref{th:convBLD}}
With an abuse of notation, we will refer to $z_{\ell+1}^k$ as the update at iteration $k$ where all the blocks have been updated until block $\ell$.
$$z_{\ell+1}^k=(w_1^{k+1},...,w_\ell^{k+1},w_{\ell+1}^k,...,w_L^k)$$
Note that $z_1^k=w^k$ and $z_{L+1}^k= w^{k+1}$.
Thanks to condition \ref{cond1} we have that
\begin{equation}\label{_seq_z}
f(z_{\ell+1}^k)\leq f(w_1^{k+1},...,w_\ell^{k}+\alpha^kd_\ell^k,...,w_L^k)\leq f{z_{\ell}^k}
\end{equation}
which implies
\begin{equation}\label{_riduco_f}
f(w^{k+1})\leq f(z_{\ell}^k)\leq f(w^k).
\end{equation}
Being $f(w)$ defined as the sum of non negative terms we have that $f(w^k)\geq 0$, and being the function coercive thanks to the regularization term $\rho \|w\|^2$, we have that its level curves are compact. Hence, being the sequence $f(w^k)$ a bounded decreasing sequence over a compact set we have that
\begin{equation}\label{_f_bar}
\lim_{k\rightarrow\infty}f(w^k)=\bar{f}.
\end{equation}
(\ref{_f_bar}) and (\ref{_seq_z}) imply that also the intermediate updates goes to the same limit
\begin{equation}\label{_f_bar_z}
\lim_{k\rightarrow\infty}f(z_{\ell}^k)=\bar{f} \qquad \forall \ell=1,...,L
\end{equation}
As a consequence, by (\ref{_seq_z}) we have that
\begin{equation}\label{armijo_a_zero}
\lim_{k\rightarrow\infty}f(z_{\ell+1}^k)-f(w_1^{k+1},...,w_\ell^k+\alpha^kd_\ell^k,w_{\ell+1}^k,...,w_L^k)=0
\end{equation}
Suppose that for a fixed index $\ell\in\lbrace1,...,L \rbrace$ $z_\ell^k$ has an accumulation point $\bar{z}_\ell$, Thanks to Armijo Linesearch we have that
\begin{equation}\label{z_bar_stazionario}
\nabla_\ell f(\bar{z}_\ell)=0 \qquad \forall \ell=1,...,L
\end{equation}
Thanks to condition (\ref{cond2}) and (\ref{_f_bar}) we have that the distance between two iterates goes to zero
\begin{equation}\label{_dist_zero}
\lim_{k\rightarrow \infty}\|z_{\ell+1}^k-z_\ell^k\|=0
\end{equation}
We can consider an accumulation point of $w^k$, $\bar{w}$. Since $w^k=z_1^k$, thanks to (\ref{_dist_zero}) we have that $\bar{w}$ is also an accumulation point for all the subsequences $z_{\ell}^k, \ell\in \lbrace 1,...,L \rbrace$. Thanks to (\ref{z_bar_stazionario}) we can conclude that $\bar{w}$ is a stationary point.

\end{proof}

\bibliographystyle{abbrv}

\bibliography{milib}   

\end{document}